\documentclass[12pt]{article}
\usepackage{amsmath, amsthm, amssymb, mathrsfs}
\usepackage{graphicx}
\usepackage[all,2cell]{xy}
\UseAllTwocells
\usepackage[all]{xy}
\usepackage{hyperref}
\usepackage{graphicx}
\usepackage{manfnt}

\font \boldfrak eufb10
\font \normalfrak eufm10

\def\fr#1{\hbox{\normalfrak #1}}
\def\bfr#1{\hbox{\boldfrak #1}}

\def\bB{{\bf B}}

\def\bG{{\bf G}}
\def\bH{{\bf H}}

\def\bM{{\bf M}}
\def\bN{{\bf N}}

\def\bS{{\bf S}}
\def\bT{{\bf T}}

\def\bZ{{\bf Z}}

\def\aq{/  \kern-.25em / }

\def\g{\mathfrak{g}}

\def\cH{\mathcal{ H}}

\def\Z{\mathbb{ Z}}

\def\boxit#1{\vbox{\hrule\hbox{\vrule\kern3pt
          \vbox{\kern3pt#1\kern3pt}\kern3pt\vrule}\hrule}}

\begin{document}

\newtheorem{theorem}{Theorem}[subsection]
\newtheorem{lemma}[theorem]{Lemma}
\newtheorem{proposition}[theorem]{Proposition}
\newtheorem{problem}[theorem]{Problem}
\newtheorem{corollary}[theorem]{Corollary}

\theoremstyle{definition}
\newtheorem{definition}[theorem]{Definition}
\newtheorem{example}[theorem]{Example}
\newtheorem{xca}[theorem]{Exercise}

\theoremstyle{remark}
\newtheorem{remark}[theorem]{\bf Remark}

\def\goth{\frak}

\def\GL{{\rm GL}}
\def\tr{{\rm tr}\, }
\def\A{{\Bbb A}}
\def\bs{\backslash}
\def\Q{{\Bbb Q}}
\def\R{{\Bbb R}}
\def\Z{{\Bbb Z}}
\def\C{{\Bbb C}}
\def\SL{{\rm SL}}
\def\cS{{\cal S}}
\def\cH{{\cal H}}
\def\G{{\Bbb G}}
\def\F{{\Bbb F}}
\def\cF{{\cal F}}

\def\cB{{\cal B}}
\def\cA{{\cal A}}
\def\cE{{\cal E}}

\newcommand{\oB}{{\overline{B}}}
\newcommand{\oN}{{\overline{N}}}

\def\CC{{\Bbb C}}
\def\ZZ{{\Bbb Z}}
\def\QQ{{\Bbb Q}}
\def\cS{{\cal S}}

\def\Ad{{\rm Ad}}

\def\bG{{\bf G}}
\def\bH{{\bf H}}
\def\bT{{\bf T}}
\def\bM{{\bf M}}
\def\bB{{\bf B}}
\def\bN{{\bf N}}
\def\bS{{\bf S}}
\def\bZ{{\bf Z}}
\def\t{\kern.1em {}^t\kern-.1em}
\def\cc#1{C_c^\infty(#1)}
\def\Fx{F^\times}
\def\half{\hbox{${1\over 2}$}}
\def\T{{\Bbb T}}
\def\Ox{{\frak O}^\times}
\def\Ex{E^\times}
\def\Ecl{{\cal E}}

\def\g{{\frak g}}
\def\h{{\frak h}}
\def\k{{\frak k}}
\def\ft{{\frak t}}
\def\n{{\frak n}}
\def\b{{\frak b}}

\def\2by2#1#2#3#4{\hbox{$\bigl( 
{#1\atop #3}{#2\atop #4}\bigr)$}}
\def\wh{\Xi}
\def\C{{\Bbb C}}
\def\bs{\backslash}
\def\adots{\mathinner{\mkern2mu
\raise1pt\hbox{.}\mkern2mu
\raise4pt\hbox{.}\mkern2mu
\raise7pt\hbox{.}\mkern1mu}}

\def\tim{\bf}
\def\timit{\it}
\def\timsm{\scriptstyle}
\def\secttt#1#2{\vskip.2in\noindent
{$\underline{\hbox{#1}}$\break (#2)}\vskip.2in}
\def\sectt#1{\vskip.2in\noindent
{$\underline{\hbox{#1}}$}\vskip.2in}
\def\sectm#1{\vskip.2in\noindent
{$\underline{#1}$}\vskip.2in}
\def\mat#1{\left[\matrix{#1}\right]}
\def\cc#1{C_c^\infty(#1)}
\def\ds{\displaystyle}
\def\Hom{\mathop{Hom}\nolimits}
\def\Ind{\mathop{Ind}\nolimits}
\def\bs{\backslash}
\def\ni{\noindent}
\def\eb{{\bf e}}
\def\fb{{\bf f}}
\def\hb{{\bf h}}
\def\rg{{\goth R}}
\def\ig{{\goth I}}
\def\ccl{{\cal C}}
\def\dcl{{\cal D}}
\def\ecl{{\cal E}}
\def\hcl{{\cal H}}
\def\ocl{{\cal O}}
\def\ncl{{\cal N}}
\def\ag{{\goth a}}
\def\bg{{\goth b}}
\def\cg{{\goth c}}
\def\og{{\goth O}}
\def\hg{{\goth h}}
\def\lg{{\goth l}}
\def\mg{{\goth m}}
\def\Og{{\goth O}}
\def\rg{{\goth r}}
\def\sg{{\goth s}}
\def\Sg{{\goth S}}
\def\tg{{\goth t}}
\def\zg{{\goth z}}
\def\C{{\Bbb C}}
\def\Q{{\Bbb Q}}
\def\R{{\Bbb R}}
\def\T{{\Bbb T}}
\def\Z{{\Bbb Z}}
\def\t{{}^t\kern-.1em}
\def\tr{\hbox{tr}}
\def\ad{{\rm ad}}
\def\Ad{{\rm Ad}}
\def\2by2#1#2#3#4{\hbox{$\bigl( {#1\atop #3}{#2\atop #4}\bigr)$}}

\def\Kcl{{\cal K}}
\def\Pcl{{\cal P}}
\def\Scl{{\cal S}}
\def\Vcl{{\cal V}}
\def\Wcl{{\cal W}}
\def\Ecl{{\cal E}}
\def\A{{\Bbb A}}
\def\F{{\Bbb F}}
\def\T{{\Bbb T}}
\def\go{{\goth O}}
\def\Ox{{\goth O}^\times}
\def\Fx{F^\times}
\def\Ex{E^\times}
\def\half{\hbox{${1\over 2}$}}
\def\vtwo{\vskip .2in}
\def\BibFH{{\bf [1]}}
\def\BibGod{{\bf [2]}}
\def\BibCrelle{{\bf [3]}}
\def\BibAmJ{{\bf [4]}}
\def\BibDuke{{\bf [5]}}
\def\BibJL{{\bf [6]}}
\def\BibRR{{\bf [7]}}

\def\BibFH{{\bf [1]}}
\def\BibGod{{\bf [2]}}
\def\BibCrelle{{\bf [3]}}
\def\BibAmJ{{\bf [4]}}
\def\BibDuke{{\bf [5]}}
\def\BibJL{{\bf [6]}}
\def\BibRR{{\bf [7]}}

\newcommand{\N}{\mathbb{N}}

\newcommand{\gen}{{\operatorname{gen}}}
\newcommand{\ind}{\operatorname{ind}}
\newcommand{\Wh}{\mathcal{W}}
\newcommand{\Kr}{\mathcal{K}}
\newcommand{\V}{\mathcal{V}}
\newcommand{\U}{\mathcal{U}}
\renewcommand{\O}{\mathcal{O}}
\newcommand{\lift}{\goth{g}}
\newcommand{\inv}{\iota}
\newcommand{\supp}{{\operatorname{Supp}}}
\def\cW{{\cal W}}

\title{Distinguished Cuspidal Representations over $p$-adic and Finite Fields}
\author{Jeffrey Hakim}
\date{\today}
\maketitle
\abstract{The author's work with Murnaghan on distinguished tame supercuspidal representations is re-examined using a simplified treatment of Jiu-Kang Yu's construction of tame supercuspidal representations of $p$-adic reductive groups.  This leads to a unification of aspects of the  theories of distinguished cuspidal representations over $p$-adic and finite fields.}
\tableofcontents

\parskip=.13in

\section{Introduction}

This paper establishes a close connection between the theories of distinguished representations over $p$-adic fields and finite fields by proving a uniform formula (Theorem \ref{maintheorem}) that was previously stated without proof in \cite{ANewYu}.

In \cite{ANewYu}, we presented a new construction of supercuspidal representations for $p$-adic reductive groups.  This construction was built on the same foundation as Yu's construction \cite{MR1824988}, but supercuspidal representations were directly associated to representations of compact-mod-center subgroups, rather than generic, cuspidal $G$-data.

From a technical point of view, the new construction simplifies Yu's construction  \cite{MR1824988} largely because it avoids the use of (noncanonical) Howe factorizations.  (See \cite[\S4.3]{MR2431732} for the relevant notion of ``factorization.'')
But perhaps the true test of the construction is how amenable it is to applications.  This paper provides the first application of the formula and hence the first basis for comparison with other approaches.

The  application considered in \S \ref{sec:newHM} is a  re-examination of the theory of distinguished tame supercuspidal representations developed (in \cite{MR2431732}) by the author and Murnaghan and we prove stronger versions of the main results with considerably less effort.

On the surface, \S\ref{sec:newHM} appears to provide a dramatically shorter treatment of the material from \cite{MR2431732}, however, part of this reduction results from the fact that substantial portions of \cite{MR2431732} are quoted in our proofs.  So it is important to acknowledge and emphasize the large influence of Fiona Murnaghan's ideas on the present paper.

In \S\ref{sec:Lusztigsec}, our main $p$-adic results are shown to have obvious analogues that are valid in the context of cuspidal Deligne-Lusztig representations over finite fields.  The results we develop in the finite field case are compatible with results of Lusztig \cite{MR1106911}.  

To get the $p$-adic and finite field theories to mesh, we articulate Lusztig's results in a new way.  In particular, we relate the character $\eta'_\theta$ occurring in \cite{MR2431732} to a character $\varepsilon_{{}_{\mathsf{T},\theta}}$ occurring in \cite{MR1106911}.
Both characters involve determinants of  adjoint actions of groups on Lie algebras over finite fields.  
In the finite field case, $\varepsilon_{{}_{\mathsf{T},\theta}}$ is computed in 
Proposition \ref{Lusztigisdetad}.
We hope to study the $p$-adic case in a subsequent paper.

Finally, we direct the reader to two preprints \cite{Fintzen-OnCon, Fintzen-TameCusp} that correct an error (discovered by Loren Spice)  in the proofs of Proposition 14.1 and Theorem 14.2 in \cite{MR1824988}.  This error affects both this paper and its precursor \cite{ANewYu}.  We  attempted to correct the error in \S3.10 of \cite{ANewYu}, but  Fintzen discovered an error in our putative correction.

\section{Statement of the main theorem}\label{sec:finitestuff}

In order not to recapitulate large amounts of notations and definitions, we follow the conventions of \cite{ANewYu}.  Practically speaking, the reader is therefore required to have a copy of \cite{ANewYu} readily accessible while reading this paper.

We consider two cases that we refer to as ``the $p$-adic case'' and ``the finite field case.''
In the $p$-adic case, $F$ is a finite extension of $\mathbb{Q}_p$ with $p$ odd.  In the finite field case, $F= \mathbb{F}_{q}$ where $q$ is a power of an odd prime $p$.  Let $\bG$ be a connected, reductive $F$-group and let $G = \bG (F)$.  (More generally, we use boldface letters for $F$-groups and  the corresponding non-bold letters for the corresponding groups of $F$-points.)

Let us selectively recall some of our inherited notations from
 \cite{ANewYu} in the $p$-adic case:
\begin{itemize}
\item $\bH$ is an $F$-subgroup of $\bG$ that is a Levi subgroup of $\bG$ over some tamely ramified finite extension of $F$, and the quotient $\bZ_\bH/\bZ_\bG$ of the centers is $F$-anisotropic.  
\item $x$ is  a vertex  in the reduced building $\mathscr{B}_{\rm red}(\bH,F)$.
\item $H = \bH (F)$.
\item $H_x$ is the stabilizer of $x$ in $H$.  The corresponding (maximal) parahoric subgroup is $H_{x,0}$ and  its prounipotent radical is  $H_{x,0+}$.
\item $\bH_{\rm sc}$ is  the universal cover of the derived group $\bH_{\rm der}$.
\item $H_{{\rm der},x,0+}^\flat$ is the image of $H_{{\rm sc},x,0+}$ in $\bH_{\rm der}$.
\end{itemize}
 
 \bigskip
Recall also that a smooth, irreducible, complex representation $(\rho ,V_\rho)$ of $H_x$ is permissible (as in Definition 2.1.1 \cite{ANewYu}) if:
\begin{itemize}
\item $\rho$ induces an irreducible (and hence supercuspidal) representation of $H$,
\item the restriction of $\rho$ to $H_{x,0+}$ is a multiple of some character $\phi$ of $H_{x,0+}$, 
\item $\phi$ is trivial on $H_{{\rm der},x,0+}^\flat$,
\item  the  dual cosets $(\phi |Z^{i,i+1}_{r_i})^*$ (defined in \S2.7 \cite{ANewYu}) contain elements that satisfy Yu's condition {\bf GE2}  (stated  in \S3.6 \cite{ANewYu}).
\end{itemize}

In the $p$-adic case, we take $L=H_x$.  In the finite field case, $L$ is the group $T = \bT (F)$ of $F$-rational points an $F$-elliptic maximal $F$-torus $\bT$ of $\bG$.

In the $p$-adic case,   $\rho$ will be a permissible representation of $L =H_x$.  
In the finite field case,  $\rho$ will be a character in general position of $L= T$.

Let $\pi (\rho)$ be the irreducible supercuspidal or cuspidal Deligne-Lusz\-tig representation of $G$ associated to $\rho$.

Let $\mathscr{I}$ be the set of $F$-automorphisms of $\bG$ of order two, and let $G$ act on $\mathscr{I}$ by
$$g\cdot \theta = {\rm Int}(g)\circ \theta \circ {\rm Int}(g)^{-1} = {\rm Int}(g\theta (g)^{-1})\circ\theta,$$ where ${\rm Int}(g)$ is conjugation by $g$.  Fix a $G$-orbit $\Theta$ in $\mathscr{I}$.

Given $\theta\in \Theta$, let $G_\theta$ be the stabilizer of $\theta$ in $G$.  Let $G^\theta$ be the group of fixed points of $\theta$ in $G$.  Let $L_\theta = G_\theta\cap L$.
When $\vartheta$ is an $L$-orbit in $\Theta$, let $${\rm m}_L (\vartheta) = [G_\theta :G^\theta L_\theta],$$ for some, hence all, $\theta\in \vartheta$.

Let $\langle \Theta,\rho\rangle_G$ denote the dimension of the space ${\rm Hom}_{G^\theta}(\pi (\rho),1)$  of $\C$-linear forms on the space of $\pi (\rho)$ that are invariant under the action of $G^\theta$ for some, hence all, $\theta\in \Theta$.

For each $\theta$ such that $\theta (L) = L$, we define a character $$\varepsilon_{{}_{L,\theta}}  :L^\theta \to \{ \pm 1\}$$ as follows.

In the finite field case,
$$\varepsilon_{{}_{L,\theta}} (h) = \prod_{a\in {\rm Gal}(\overline{F}/F) \bs \Phi (Z_{\bG}((\bT^\theta)^\circ ) ,\bT)} a(h).$$
In other words, $\varepsilon_{{}_{L,\theta}} (h)$ is $(-1)^s$, where $s$ is the number of Galois orbits of roots $a$ in $\Phi (Z_{\bG}((\bT^\theta)^\circ ) ,\bT)$ such that $a(h)=-1$.

In the $p$-adic case,
$$\varepsilon_{{}_{L,\theta}} (h) = \prod_{i=0}^{d-1} \left(\frac{\det\nolimits_{{\frak f}} ({\rm Ad}(h)\, |\, {\frak W}_i^+)}{\fr{P}_F}\right)_2,$$
with the following notations.

First, we let
$$\mathfrak{W}_i^+
=\bigg(\bigg(\bigoplus_{a\in \Phi^{i+1}-\Phi^i}\bfr{g}_a\bigg)^{{\rm Gal}(\overline{F}/F)}\bigg)^\theta_{x,s_i:s_i+},$$ viewed as a vector space over the residue field $\fr{f}$ of $F$, where $\bfr{g}_a$ is the root space attached to the root $a$.
So $\mathfrak{W}_i^+$  may be viewed as the space of $\theta$-fixed points in the Lie algebra of $W_i= J^{i+1}/J^{i+1}_+$.   (Here, ``the Lie algebra of $W_i$'' really means the image of $W_i$ under a suitable Moy-Prasad isomorphism.)

Next, for $u\in \fr{f}^\times$, we let $(u/\fr{P}_F)_2$ denote the quadratic residue symbol.  This is related to the ordinary Legendre symbol   by
$$\left(\frac{u}{\fr{P}_F}\right)_2
=\left(\frac{N_{\fr{f}/\mathbb{F}_p} (u)}{p}\right) = (N_{\fr{f}/\mathbb{F}_p} (u))^{(p-1)/2}= u^{(q_F-1)/2}.$$

 This is the same as the character $\eta'_\theta$ defined in \cite[\S5.6]{MR2431732}, but we have expressed it  on the Lie algebra.



When $\vartheta$ is an $L$-orbit in $\Theta$, we write $\vartheta \sim \rho$ if $\theta (L) = L$ and if  
the space ${\rm Hom}_{L^\theta} (\rho ,\varepsilon_{{}_{L,\theta}})$ is nonzero for some, hence all, $\theta\in \vartheta$.
When $\vartheta\sim \rho$, we define
$$\langle \vartheta ,\rho\rangle_L = \dim {\rm Hom}_{L^\theta} (\rho ,\varepsilon_{{}_{L,\theta}}  ),$$ where $\theta$ is any element of $\vartheta$.  (The choice of $\theta$ does not matter.)

We can now state our main theorem:

\medskip

\begin{theorem}\label{maintheorem}
$\langle \Theta,\rho\rangle_G = \sum_{\vartheta\sim \rho} {\rm m}_L(\vartheta)\, \langle \vartheta , \rho\rangle_L$.
\end{theorem}

\medskip
In the finite field case, this is contained in
Proposition \ref{mainfiniteprop} and it is further refined in Theorem \ref{newLusztigthm}.
In the $p$-adic  case, it is contained in
Proposition \ref{mainpadicprop}.

Note that in  the special case in which
 \begin{itemize}
\item $\bG$ is a product $\bG_1\times\bG_1$,
\item $\Theta$ contains the involution $\theta (x,y) = (y,x)$, 
\item $\rho$ has the form $\rho_1\times \tilde\rho_2$, where $\tilde\rho_2$ is the contragredient of $\rho_2$,
\end{itemize}
we have
$$\langle \Theta,\rho\rangle_G = \dim {\rm Hom}_{G}(\pi (\rho_1),\pi(\rho_2)).$$
In the finite field case, this is equivalent to   the  Deligne-Lusztig inner product formula \cite[Theorem 6.8]{MR0393266}.  See \cite[page 58]{MR1106911}, for more details.

\section{$p$-adic representation theory}\label{sec:newHM}

\subsection{$\theta$-symmetry}\label{sec:symmetry}

The present paper should be viewed as a sequel to \cite{ANewYu} and for the $p$-adic theory we use precisely the same notation, terminology, and conventions.
As in \cite{ANewYu}, we fix all of the following objects:
\begin{itemize}
\item $F$ : a finite extension of $\Q_p$, with $p\ne 2$,
\item $\bG$ : a connected reductive $F$-group,
\item $\bH$ : an $F$-subgroup of $\bG$ that is a Levi subgroup over some tamely ramified finite extension of $F$ such  that the quotient $\bZ_\bH/\bZ_\bG$ of the centers  is $F$-anisotropic,
\item $x$ : a vertex in the reduced building $\mathscr{B}_{\rm red}(\bH,F)$,
\item $(\rho , V_\rho)$ : a permissible representation of $H_x$,
\item $(\pi ,V_\pi)$ : a supercuspidal representation of $G = \bG (F)$ in the isomorphism class associated to $\rho$.
\end{itemize}

We refer to $F$-automorphisms of $\bG$ of order two as {\it involutions of $G$}, and we let $G$ act on its set of involutions by
$$g\cdot \theta = {\rm Int}(g)\circ \theta \circ {\rm Int}(g)^{-1}.$$
For the rest of this chapter, we assume we have fixed a $G$-orbit $\Theta$ of involutions of $G$.

We define
$$\langle \Theta , \rho \rangle_G = \dim {\rm Hom}_{G^\theta} (\pi ,1),$$
where $\theta$ is any element of $\Theta$ and $\bG^\theta$ is the group of fixed points of $\theta$ in $\bG$.  The fact that this definition is independent of the choice of $\theta$ is a consequence of the fact that
we have a bijection
$$\xymatrix{{\rm Hom}_{G^\theta} (\pi ,1)\ar[r]^{\kern-.5em\simeq}&{\rm Hom}_{G^{g\cdot \theta}} (\pi ,1)\\
\lambda\ar@{|->}[r]&\left( v\mapsto \lambda (\pi (g)^{-1}v)\right),}$$
for each $g\in G$.

It is elementary to show that if $K$ is the open, compact-mod-center inducing group for $\pi$ then
$$\langle \Theta, \rho \rangle_G = \sum_{\mathscr{O}\in \Theta^K} {\rm m}_K(\mathscr{O})\, \langle \mathscr{O},\rho\rangle_K,$$
where:
\begin{itemize}
\item $\Theta^K$ is the set of $K$-orbits in $\Theta$,
\item ${\rm m}_K(\mathscr{O}) = [G_\theta :G^\theta K_\theta]$, where $\theta$ is any element of $\mathscr{O}$, $G_\theta$ is the stabilizer of $\theta$ in $G$, and $K_\theta =K\cap G_\theta$,
\item $\langle \mathscr{O},\rho\rangle_K = \dim {\rm Hom}_{K^\theta}(\kappa ,1)$, for any $\theta\in \mathscr{O}$, and $K^\theta = K\cap G^\theta$, where $\kappa$ is the representation (see \S3.11 \cite{ANewYu}) of $K$ from which $\pi$ is induced.
\end{itemize}
We refer to  \cite[\S3.1]{MR2925798} for an explanation of the details, including the facts that the definitions of ${\rm m}_K(\mathscr{O})$ and $\langle \mathscr{O},\rho\rangle_K$ do not depend on the choice of $\theta$ in $\mathscr{O}$.  We also note that ${\rm m}_K (\mathscr{O})$ is a power of two that is bounded as indicated in  \cite[\S3.1.2]{MR2925798}.

Recall from \cite[\S2.4]{ANewYu} that we have a tamely ramified twisted Levi sequence $\vec\bG = (\bG^0,\dots ,\bG^d)$ associated to $\rho$.

The purpose of this section is to prove:

\begin{proposition}\label{summaryofsection}
Suppose $\mathscr{O}$ is a $K$-orbit of involutions of $G$ such that $\langle\mathscr{O}, \rho\rangle_K$ is nonzero.   Then:
\begin{itemize}
\item[{\rm (1)}] There exists $\theta\in \mathscr{O}$ such that $\theta (\vec\bG) = \vec\bG$.
\item[{\rm (2)}]  Such an involution $\theta$ must fix $x$.
\item[{\rm (3)}]  The character $\hat\phi$ (defined in \S3.9 \cite{ANewYu}) must be trivial on $K^\theta_+$ (defined in \S2.6 \cite{ANewYu}), and hence $\phi$ must be trivial on $H_{x,0+}^\theta$.
\item[{\rm (4)}] For each $i\in \{ 0,\dots, d-1\}$, there must exist an element $X^*_i$ in the dual coset $(\phi|Z^{i,i+1}_{r_i})^*$ such that $\theta (X^*_i) = -X^*_i$.
\item[{\rm (5)}] Up to scalar multiples, there exists a canonical isomorphism $${\rm Hom}_{K^\theta} (\kappa, 1) \cong {\rm Hom}_{H_x^\theta}(\rho,\varepsilon_{{}_{H_x,\theta}}).$$
This isomorphism is defined below and it still exists if we replace the hypothesis $\langle\mathscr{O}, \rho\rangle_K\ne 0$ with the weaker hypotheses that $\hat\phi|K^\theta_+=1$ and $\theta (\vec\bG) = \vec\bG$.
\item[{\rm (6)}] $\langle\mathscr{O}, \rho\rangle_K= \dim {\rm Hom}_{H_x^\theta}(\rho,\varepsilon_{{}_{H_x,\theta}})$.
\end{itemize}
\end{proposition}

For the rest of this section, we state and prove a sequence of lemmas that collectively encompass Proposition \ref{summaryofsection}.

The first lemma is an analogue of Lemma 5.15 \cite{MR2431732}, but the proof is much simpler.

\bigskip
\begin{lemma}
If $\mathscr{O}$ is a $K$-orbit of involutions of $G$ and $\hat\phi |K^\theta_+=1$ for all $\theta\in \mathscr{O}$ then there exists $\theta\in \mathscr{O}$ such that $\theta (\vec\bG) = \vec\bG$.
\end{lemma}

\begin{proof}
There is nothing to prove if $d=0$, so we assume $d>0$.

Starting with an arbitrary element $\theta_d$ of $\mathscr{O}$, we recursively construct a sequence $\theta_d,\dots , \theta_0$ of elements of $\mathscr{O}$ such that $\theta_i (\bG^j) = \bG^j$ whenever $i\le j\le d$.

Assume $\theta_{i+1}$ has already been constructed.  On page 52 \cite{ANewYu}, we define groups $J^{i+1}$ and $J^{i+1}_+$.
Since  $\hat\phi |K^{\theta_{i+1}}_+=1$, we deduce that $\hat\phi$ is trivial on $G^{i+1}_{\rm der}\cap J^{i+1}_+\cap G^{\theta_{i+1}}$.  But on the latter subgroup, and in fact on $G^{i+1}_{\rm der}\cap J^{i+1}_+$, the character $\hat\phi$ is represented by each element $X^*_i$ of the dual coset $(\phi |Z^{i,i+1}_{r_i})^*$ (defined in \cite[\S2.7]{ANewYu}).  (See Lemma 3.9.1(5) \cite{ANewYu}.)

Therefore, for a given $X^*_i$, $$1= \hat\phi (\exp (X+\fr{g}^{i+1}_{x,r_i+})) = \psi (X^*_i (X)),$$
for all $X\in \fr{g}^{i+1}_{\rm der}\cap \fr{J}^{i+1}_+\cap \fr{g}^{\theta_{i+1}}$.
Here, ``exp'' refers to the isomorphism
$$\exp :\fr{J}^{i+1}_+/\fr{g}^{i+1}_{x,r_i+}\to J^{i+1}_+/G^{i+1}_{x,r_i+}$$ as in \cite[Corollary 2.4]{MR1824988}.  (See also \cite[\S3.1]{ANewYu}.)  We are  using  the abbreviation $\theta_{i+1}$ for $d\theta_{i+1}$.  (More details on our choice of the additive character $\psi$ and its role in the definition of the dual coset are given in \S2.7 \cite{ANewYu}.)

It follows that $X_i^*\in (\fr{g}^{i+1}_{\rm der}\cap \fr{J}^{i+1}_+\cap \fr{g}^{\theta_{i+1}})^\bullet$, where, as in the proof of Lemma 5.15 \cite{MR2431732}, we let $$\fr{s}^\bullet = \{ Y^*\in \fr{g}^{i+1,*}\, :\, Y^*(\fr{s})\subset \fr{P}_F\},$$ when $\fr{s}$ is a subset of $\fr{g}^{i+1}$.

We have $$ (\fr{g}^{i+1}_{\rm der}\cap \fr{J}^{i+1}_+\cap \fr{g}^{\theta_{i+1}})^\bullet =  \fr{g}^{i+1,\bullet}_{\rm der}+ \fr{J}^{i+1,\bullet}_+ + \fr{g}^{\theta_{i+1},\bullet},$$
and
\begin{itemize}
\item[$\circ$] $ \fr{g}^{i+1,\bullet}_{\rm der} = \fr{z}^{i+1,*}$,
\item[$\circ$] $\fr{J}^{i+1,\bullet}_+ = (\fr{g}^{i,*},\fr{g}^{i+1,*})_{x,((-r_i)+,-s_i)}$,
\item[$\circ$] $\fr{g}^{\theta_{i+1},\bullet} = \{ Y^*\in \fr{g}^{i+1,*}\, :\, \theta_{i+1}(Y^*) = -Y^*\}$.
\end{itemize}
Therefore, we can choose $Y^*\in \fr{J}^{i+1,\bullet}$ and $Z^* \in \fr{z}^{i+1,*}$ such that $X^*_i+Y^*+ Z^*\in \fr{g}^{\theta_{i+1},\bullet}$.
Since $X^*_i+Z^*$ is $\bG^{i+1}$-generic and since $G^{i+1}_{x,s_i} = J^{i+1} G^i_{x,s_i}$, we deduce from Lemma 8.6 \cite{MR1824988} that
$$X^*_i +Z^* + \fr{J}^{i+1,\bullet}_+ = \Ad^*(J^{i+1})(X^*_i +Z^* + \fr{g}^{i,*}_{x,(-r_i)+}).$$
Therefore, we can choose $k\in J^{i+1}$ and $U^*\in \fr{g}^{i,*}_{x,(-r_i)+}$ such that
$$X^*_i +Y^*+ Z^*  = \Ad^*(k)(X^*_i +Z^* + U^*).$$

We take $\theta_i = k^{-1}\cdot \theta_{i+1}$ and observe that $$\theta_i (X^*_i + Z^*+U^*) = - (X^*_i +Z^*+U^*).$$  Lemma 5.17 \cite{MR2431732} implies that $\theta_i (\bG^i) = \bG^i$.  In addition, $\theta_i (\bG^j) = \bG^j$ for $i< j\le d$.

This completes the recursion, and taking $\theta = \theta_0$ completes the proof.
\end{proof}

\begin{lemma}\label{dualcosetinvariance}
If $\hat\phi |K^\theta_+=1$, $\theta (\vec\bG) = \vec\bG$, and $i\in \{ 0,\dots , d-1\}$ then the dual coset $(\phi |Z^{i,i+1}_{r_i})^*$ contains elements $X^*_i\in \fr{z}^{i,i+1,*}_{-r_i}$ such that $\theta (X^*_i ) =-X^*_i$.
\end{lemma}

\begin{proof}
Choose an arbitrary element $U^*_i$ in the dual coset.
It suffices to show that $U^*_i$ represents the same character of $\fr{z}^{i,i+1}_{r_i}$ as the element
$$X^*_i = \frac{U^*_i -\theta(U^*_i)}{2}.$$

We first observe that for each $Y\in \fr{z}^{i,i+1}_{r_i}$ the associated element $(Y+\theta (Y))/2$ lies in the space $(\fr{z}^{i,i+1}_{r_i})^\theta$ of $\theta$-fixed elements of $\fr{z}^{i,i+1}_{r_i}$.
Hence $\exp \left( \frac{Y+\theta(Y)}{2} + \fr{z}^{i,i+1}_{r_i+}\right)$ lies in $(Z^{i,i+1}_{r_i:r_i+})^\theta$ or, equivalently, the image of $(Z^{i,i+1}_{r_i})^\theta$ in $Z^{i,i+1}_{r_i:r_i+}$.

Thus $$\psi \left( U^*_i \left(\frac{Y+\theta(Y)}{2} \right)\right)
=\phi \left( \exp  \left(\frac{Y+\theta(Y)}{2}+ \fr{z}^{i,i+1}_{r_i+} \right)\right)=1$$
and therefore
$$\psi \left( U^*_i \left(\frac{Y}{2} \right)\right)
=\psi \left( U^*_i \left(\frac{\theta(Y)}{2} \right)\right)^{-1}
=\psi \left( \theta (U^*_i) \left(\frac{Y}{2} \right)\right)^{-1}.$$
Hence
$$\psi \left(\frac{U^*_i+\theta (U^*_i)}{2} (Y)\right)=1.$$  Our claim  follows.
\end{proof}

We now adapt Proposition 4.2 from \cite{MR2431732} whose proof is rather involved.  In the statement, $(\tau_i,V_i)$ is the Heisenberg represenattion of the group $\mathscr{H}_i = W_i\times \mu_p$ defined in \S2.8 \cite{ANewYu}.
The notations $\omega_i$ and $\mathscr{S}_i$ are also defined in \S2.8 \cite{ANewYu}.

\begin{lemma}\label{newpropfourtwo}
If $\hat\phi |K^\theta_+=1$, $\theta (\vec\bG) = \vec\bG$, and $i\in \{ 0,\dots , d-1\}$
 then the spaces ${\rm Hom}_{W^+_i\times 1}(\tau_i,1)$ and $V_i^{\tau_i(W_i^+\times 1)}$ have dimension one, where
$$W^+_i = J^{i+1,\theta}J^{i+1}_+/J^{i+1}_+.$$
If $$\mathscr{P}_i = \{ s\in \mathscr{S}_i\ :\ sW^+_i\subseteq W^+_i\}$$ and $\varepsilon_i$ is the unique character of $\mathscr{P}_i$ of order two then
$${\rm Hom}_{W^+_i\times 1}(\tau_i,1)\subseteq {\rm Hom}_{K^{i,\theta}}(\omega_i,\varepsilon_i).$$
\end{lemma}

\begin{proof}
Our assertions follow directly from 
Proposition 4.2 \cite{MR2431732} where we replace $(\bG',\bG,\phi |G'_{y,r})$ by $(\bG^i,\bG^{i+1},\zeta_i |G^i_{x,r_i} )$.  It should be noted that in \cite{MR2431732} 
one assumes that one has a quasicharacter of $G'$ that is $G$-generic of (positive) depth $r$.  However, a line-by-line analysis of the proof of Proposition 4.2 \cite{MR2431732} reveals that the proof only uses the restriction of the latter quasicharacter to $G'_{y,r}$ and, moreover, there is no need to require that this restriction extends to a quasicharacter of $G'$.
\end{proof}

\begin{lemma}\label{Heisinvariance}
If $\hat\phi |K^\theta_+=1$, $\theta (\vec\bG) = \vec\bG$, and $i\in \{ 0,\dots , d-1\}$ then $$V^{\kappa (J^{i+1,\theta})} = V_\rho\otimes V_0\otimes \cdots \otimes V_{i-1}\otimes V_i^{\tau_i (W_i^+\times 1)}\otimes V_{i+1}\otimes \cdots \otimes V_{d-1}.$$
\end{lemma}

\begin{proof}
The representation $\kappa$ can be constructed as in \S3.11 \cite{ANewYu}.  
The construction requires the choice  of a special homomorphism $\nu_i$ (in the sense of Definition 3.8.1 \cite{ANewYu}) and a character $\chi_i$ of $J^{i+1}$ (as described in \S3.11 \cite{ANewYu}).
These choices (within the restrictions of \cite{ANewYu}) do not affect the isomorphism class of $\kappa$, so we will make choices that are most convenient for our present purposes.

In particular, we choose $\nu_i$ so that it is associated to Yu's special isomorphism $\nu_i^\bullet$ (see Definition 3.15 \cite{MR2431732}) in the sense that the following diagram commutes:
$$\xymatrix{J^{i+1}\ar[r]^{\nu_i\quad}\ar[d]&\mathscr{H}_i= W_i\boxtimes \mu_p\ar[d]^{{\rm Id}\times\zeta_i^{-1}}\\
J^{i+1}/\ker \zeta_i\ar[r]^{\nu_i^\bullet\qquad}&W_i\boxtimes (J^{i+1}_+/\ker \zeta_i).}$$  (The $\boxtimes$ notation is explained in \S2.3 \cite{MR2431732}.)
With this choice of $\nu_i$, Proposition 4.2 \cite{MR2431732} implies $\nu_i (J^{i+1,\theta}) = W_i^+\times 1$.

The character $\chi_i$ is chosen as follows.  As in \S3.11 \cite{ANewYu}, define
$$J^{i+1}_\flat = \bG^{i+1}_{\rm der}\cap J^{i+1},\qquad 
J^{i+1}_{\flat , +} = \bG^{i+1}_{\rm der}\cap J^{i+1}_+.$$
Then $J^{i+1}/(J^{i+1,\theta}J^{i+1}_\flat)$ is a compact abelian group and we may view $\zeta_i^{-1}(\hat\phi | J^{i+1}_+)$ as a character of the subgroup
$J^{i+1}_+/(J^{i+1,\theta}_+J^{i+1}_{\flat ,+})$.  (Here, we are using Lemma \ref{dualcosetinvariance}.  We also caution that, unlike in the diagram above, the notation $\zeta_i^{-1}$ does not denote the inverse function of $\zeta_i$, but rather $\zeta_i^{-1}(z) =\zeta_i(z)^{-1}$.)

We take $\chi_i$ to be a character of the compact abelian group $$J^{i+1}/(J^{i+1,\theta}J^{i+1}_\flat)$$ that extends the character $\zeta_i^{-1}(\hat\phi | J^{i+1}_+)$ of the subgroup
$$J^{i+1}_+/(J^{i+1,\theta}_+J^{i+1}_{\flat ,+}).$$

With these choices, if $j\in J^{i+1,\theta}$ then
$$\kappa (j) = 1_\rho \otimes 1_0\otimes \cdots 1_{i-1}\otimes \tau_i(\nu_i(j))\otimes 1_{i+1}\otimes \cdots \otimes 1_{d-1},$$ according to the 
construction in \S3.11 \cite{ANewYu}.  This implies that our assertion holds.
\end{proof}

We are interested in studying the space ${\rm Hom}_{K^\theta}(\kappa, 1)$.  As a preliminary step, we study the space ${\rm Hom}_{J^{1,\theta}\cdots J^{d,\theta}}(\kappa ,1)$ of linear forms on $V$ that are fixed by each of the groups $J^{1,\theta},\dots , J^{d,\theta}$, and its subspace ${\rm Hom}_{H_x^\theta J^{1,\theta}\cdots J^{d,\theta}}(\kappa ,1)$ of $H^\theta_x$-fixed linear forms.

In the following discussion, the reader should consult \S3.11 \cite{ANewYu} for basic facts about the construction of $\kappa$ and the relation of $\kappa$ to auxiliary objects such as $\rho$ and $\omega_i$.

We observe now that the character $\varepsilon_{{}_{H_x,\theta}}$ defined in \S\ref{sec:finitestuff} may also be expressed as
$$\varepsilon_{{}_{H_x,\theta}} = \prod_{i=0}^{d-1} \left(\varepsilon_i\big|H_x^\theta\right)$$
of $H^\theta_x$, and we note that $\varepsilon_{{}_{H_x,\theta}}^2=1$.  (See   5.5 \cite{MR2431732}.)

\begin{lemma}\label{ithfactor}
Suppose  $\hat\phi |K^\theta_+=1$ and $\theta (\vec\bG) = \vec\bG$, and for each $i\in \{ 0,\dots , d-1\}$ choose a nonzero element $v_i^\circ$ in the 1-dimensional space $V_i^{\tau_i (W_i^+\times 1)}$.  Then
$$\lambda\mapsto \left( v_\rho\mapsto \lambda (v_\rho\otimes v_0^\circ\otimes \cdots \otimes v_{d-1}^\circ)\right)$$ determines a linear isomorphism
$${\rm Hom}_{J^{1,\theta}\cdots J^{d,\theta}}(\kappa ,1)\cong  {\rm Hom}_\C (V_\rho ,\C)$$ that is canonical up to scalar multiples.
The latter isomorphism restricts to an isomorphism
$${\rm Hom}_{H_x^\theta J^{1,\theta}\cdots J^{d,\theta}}(\kappa ,1)\cong {\rm Hom}_{H_x^\theta} (\rho ,\varepsilon_{{}_{H_x,\theta}}  ).$$
\end{lemma}

\begin{proof}
According to Lemma \ref{Heisinvariance}, we have an isomorphism
$$V_\rho\cong V^{\kappa (J^{1,\theta}\cdots J^{d,\theta})} :v_\rho \mapsto v_\rho \otimes v_0^\circ\otimes \cdots \otimes v_{d-1}^\circ.$$
We also have a projection $V\to V^{\kappa (J^{1,\theta}\cdots J^{d,\theta})} $ that on elementary tensors $v_\rho\otimes v_0\otimes \cdots \otimes v_{d-1}$ replaces each factor $v_i$ (other than $v_\rho$) with the average of its $\tau_i (W_i^+\times 1)$-translates.
Thus we obtain a projection $V\to V_\rho$.
A linear form on $V$ is invariant under $J^{1,\theta}\cdots J^{d,\theta}$ precisely when it factors through this  projection $V\to V_\rho$.

%

Our assertion that 
$${\rm Hom}_{J^{1,\theta}\cdots J^{d,\theta}}(\kappa ,1)\cong {\rm Hom}_\C (V_\rho ,\C)$$
now follows.
The assertion that $${\rm Hom}_{H_x^\theta J^{1,\theta}\cdots J^{d,\theta}}(\kappa ,1)\cong {\rm Hom}_{H_x^\theta} (\rho ,\varepsilon_{{}_{H_x,\theta}} )$$
follows from
Lemma \ref{newpropfourtwo}.
\end{proof}

\begin{lemma}\label{thetafixesx} If  $\hat\phi |K^\theta_+=1$ and $\theta (\vec\bG) = \vec\bG$ then $\theta x = x$.\end{lemma}

\begin{proof}
Our proof is modeled after the proof of Proposition 5.20 \cite{MR2431732} and makes use of facts deduced within the latter proof.

Suppose that $\theta x\ne x$.  There exists an apartment  in $\mathscr{B}_{\rm red}(\bH, F)$ that contains $x$ and $\theta (x)$, and within such an apartment we may choose a point $z$
such that $x$ lies on the boundary of the facet of $z$ in $\mathscr{B}_{\rm red}(\bH, F)$.

Over the residue field $\fr{f}$ of $F$, we have a reductive group $\mathsf{H}_x$ and a proper parabolic subgroup $\mathsf{P}$ with unipotent radical $\mathsf{U}$, such that $$\mathsf{H}_x(\fr{f}) = H_{x,0:0+},\quad \mathsf{P}(\fr{f}) = H_{z,0}/H_{x,0+},\quad \mathsf{U}(\fr{f}) = H_{z,0+}/H_{x,0+}.
$$

It is shown in the proof of Proposition 5.20 \cite{MR2431732} that
$$H_{z,0+} = H^\theta_{z,0+}H_{x,0+}.$$
Similarly, replacing $\bH$ by $\bH_{\rm der}$, we obtain
$$H_{{\rm der},z,0+} = H^\theta_{{\rm der}, z,0+}H_{{\rm der}, x,0+}.$$

Using the first of the latter two decompositions, as well as the fact that the character $\phi :H_{x,0+}\to \C^\times$ is trivial on $H^\theta_{x,0+}$, we see that we can inflate $\phi$ over $H^\theta_{z,0+}$ to obtain a character ${\rm inf}_{H_{x,0+}}^{H_{z,0+}}(\phi)$.

We can assume, after passing to a $z$-extension, that $p$ does not divide the order of the fundamental group of $\bH_{\rm der}$.  This allows us to apply Lemma 3.2.1 \cite{ANewYu} and the definition of permissibility to see that ${\rm inf}_{H_{x,0+}}^{H_{z,0+}}(\phi)$ extends to a quasicharacter $\chi$ of $H$.
Here, we use the decomposition  $H_{{\rm der},z,0+} = H^\theta_{{\rm der}, z,0+}H_{{\rm der}, x,0+}$ to observe that ${\rm inf}_{H_{x,0+}}^{H_{z,0+}}(\phi)$ is trivial on $H_{{\rm der},z,0+}$.  But, by Lemma 3.2.1 \cite{ANewYu}, $H_{{\rm der},z,0+}$ is identical to $[H,H]\cap H_{z,0+}$.

Given $\chi$, we let $\rho_0 = \rho\otimes (\chi|H_x)^{-1}$.  Then $\rho_0$ has depth zero and, according to Corollary 3.3.3 \cite{ANewYu}, it induces an irreducible (supercuspidal) representation of $H$.

The restriction of $\rho_0$ factors to a (possibly reducible) cuspidal representation $\bar\rho_0$ of $\mathsf{H}_x(\fr{f})$.  Cuspidality implies that
$${\rm Hom}_{H^\theta_{z,0+}}(\rho_0 , 1) =
{\rm Hom}_{\mathsf{U}(\fr{f})}(\bar\rho_0 , 1) = 0.$$

But this yields the following contradiction:
$${\rm Hom}_{H_x^\theta}(\rho,\varepsilon_{{}_{H_x,\theta}}  )\subseteq {\rm Hom}_{H^\theta_{z,0+}}(\rho_0 , 1)  = 0.$$
Note that we have used the fact that, by construction, $\chi$ is trivial on $H^\theta_{z,0+}$.  Moreover, we have used the fact that $\varepsilon_{{}_{H_x,\theta}}$ is also trivial on $H^\theta_{z,0+}$.  This follows from an argument  as in the proof of Proposition 5.20 \cite{MR2431732}.
\end{proof}

\medskip

\begin{corollary}\label{corktheta} If  $\hat\phi |K^\theta_+=1$ and $\theta (\vec\bG) = \vec\bG$ then
$$K^\theta = H_x^\theta J^{1,\theta}\cdots J^{d,\theta}.$$
\end{corollary}

\begin{proof}
This result is a variant of Proposition 3.14 \cite{MR2431732} and it follows from  the results cited in the proof of the latter result  and Lemma \ref{thetafixesx} above.  
\end{proof}

\begin{lemma}
If $\mathscr{O}$ is a $K$-orbit of involutions of $G$ such that $\langle\mathscr{O}, \rho\rangle_K$ is nonzero then the character $\hat\phi$  must be trivial on $K^\theta_+$, and hence $\phi$ must be trivial on $H_{x,0+}^\theta$.
\end{lemma}

\begin{proof}
Our claims follow from the fact that $\kappa |K_+$ is a multiple of $\hat\phi$ (according  Theorem 2.8.1(2) \cite{ANewYu}) and the fact that, by definition,  $\hat\phi |H_{x,0+} = \phi$.
\end{proof}

\subsection{From $K$-orbits of involutions to $H_x$-orbits of involutions}

\medskip
\begin{definition}
An orbit $\mathscr{O}\in \Theta^K$ is {\bf relevant} if $ \langle \mathscr{O},\rho\rangle_K$ is nonzero.
\end{definition}

\begin{definition} An involution $\theta$ of $G$ is {\bf stabilizing} if $\theta (\vec\bG) = \vec\bG$ and $\theta x= x$.
\end{definition}

Proposition \ref{summaryofsection} implies that every relevant orbit contains a stabilizing involution.

\begin{lemma}\label{stabilizing}
If $\mathscr{O}\in \Theta^K$ is relevant then $H_x$ acts transitively on the set of stabilizing involutions in $\mathscr{O}$.
\end{lemma}

\begin{proof}
Fix a stabilizing involution $\theta$ in $\mathscr{O}$.
Clearly, every element of the $H_x$-orbit of $\theta$ is stabilizing.
Now suppose we are given a stabilizing involution in $\mathscr{O}$.  Then this involution must have the form $k\cdot \theta$, for some $k\in K$.  Then $k\cdot \theta$ must stabilize $H_x$ and thus, according to Proposition 3.7 \cite{MR2925798}, $k$ must lie in $H_x$.  This proves our assertion.
\end{proof}

Our next result is an adaptation of Lemma 3.3 \cite{MR2925798}.

\begin{lemma}\label{somedecomp}
If $\theta$ is a stabilizing involution then $$K_\theta = H_{x,\theta}J^{1,\theta}\cdots J^{d,\theta},$$ where $H_{x,\theta} = H_x\cap G_\theta$.
\end{lemma}

\begin{proof}
The desired result follows from repeatedly applying Lemma 3.4 \cite{MR2925798}, however, we note that the statement of the latter result is missing the hypothesis $H^1_\theta (A\cap B) =1$ used in the proof.

More precisely, one shows
$$K_\theta = (H_xJ^1\cdots J^{d-1})_\theta J^{d,\theta} =\cdots = H_{x,\theta}J^{1,\theta}\cdots J^{d,\theta}.$$
To see that the hypotheses of Lemma 3.4 \cite{MR2925798} are satisfied, we refer to Proposition 2.12 \cite{MR2431732} and Lemma 3.12 \cite{MR2431732}.
\end{proof}

When $\theta$ is a stabilizing involution and $\vartheta$ is its $H_x$-orbit then we define
$${\rm m}_{H_x} (\vartheta) = [ G_\theta :G^\theta H_{x,\theta}],$$ where $H_{x,\theta} = H_x\cap G_\theta$.  It is easy to see that this definition does not depend on the choice of $\theta$ in $\vartheta$.

\begin{lemma}\label{midentity}
If $\theta$ is a stabilizing involution then 
$${\rm m}_{H_x} (H_x \cdot \theta) = {\rm m}_{K}(K\cdot \theta).$$
\end{lemma}

\begin{proof}
This follows directly from the definitions and Lemma \ref{somedecomp}.
\end{proof}

Recall that if $\theta\in \Theta$ has $H_x$-orbit $\vartheta$ then we have defined
$$\langle \vartheta,\rho\rangle_{H_x} = \dim {\rm Hom}_{H_x^\theta}(\rho,\varepsilon_{{}_{H_x,\theta}}),$$ and we note that this definition does not depend on the choice of $\theta$ in $\vartheta$.
We also   write $\vartheta \sim\rho$ when $H_x$ is $\theta$-stable for all $\theta \in \vartheta$ and, in addition, $\langle \vartheta,\rho\rangle_{H_x}$ is nonzero.

\begin{proposition}\label{mainpadicprop}
$$\langle \Theta, \rho \rangle_G = \sum_{\vartheta\sim \rho} {\rm m}_{H_x}(\vartheta )\, \langle \vartheta,\rho\rangle_{H_x}.$$
\end{proposition}

\begin{proof}
Recall that
$$\langle \Theta, \rho \rangle_G = \sum_{\mathscr{O}\in \Theta^K} {\rm m}_K(\mathscr{O})\, \langle \mathscr{O},\rho\rangle_K.$$
Suppose $\langle \mathscr{O},\rho\rangle_K$ is nonzero or, in other words, $\mathscr{O}$ is relevant.  Then according to Lemma \ref{stabilizing}, $\mathscr{O}$ contains a unique $H_x$-orbit $\vartheta$ that consists of all the stabilizing involutions in $\mathscr{O}$.    Lemma \ref{midentity} implies that ${\rm m}_K(\mathscr{O}) = {\rm m}_{H_x} (\vartheta)$ and Proposition \ref{summaryofsection}(5) implies that $\langle \mathscr{O},\rho\rangle_K= \langle \vartheta,\rho\rangle_{H_x}$.  Proposition \ref{summaryofsection} also implies $\theta (H_x) = H_x$ and hence $\vartheta \sim \rho$.

It remains to show that if $\vartheta$ is any $H_x$-orbit in $\Theta$ such that $\vartheta \sim \rho$ then its $K$-orbit $\mathscr{O}$ is relevant.  
Given such an orbit $\vartheta$, fix $\theta\in \vartheta$.  Lemma 3.5 \cite{MR2925798} implies that $\theta$ stabilizes $\bH$ and $x$.
According to Lemma \ref{ithfactor} and Corollary \ref{corktheta}, it now suffices to show that $\hat\phi |K^\theta_+=1$ and $\theta (\vec\bG) = \vec\bG$.

One can inductively show that the groups $\bG^d,\dots ,\bG^0$ are $\theta$-stable by using the $\theta$-stability of $\bH$ and $x$ together with the fact that  $\bG^i$ is defined (in \S2.4 \cite{ANewYu}) to be the (unique) maximal subgroup of $\bG$ such that:
\begin{itemize}
\item $\bG^i$ is defined over $F$,
\item $\bG^i$ contains $\bH$,
\item $\bG^i$ is a Levi subgroup of $\bG$ over $E$,
\item $\phi |H^{i-1}_{x,r_i}=1$, where $\bH^{i-1} = \bG^i_{\rm der}\cap \bH$.
\end{itemize}

Finally, we observe that $\varepsilon_{{}_{H_x,\theta}} |H_{x,0+}^\theta =1$ since it is a character of exponent 2 of a pro-$p$-group with $p$ odd.  Therefore, $\vartheta \sim \rho$ implies that $\phi |H_{x,0+}^\theta=1$.  Since $\theta (\vec\bG) = \vec\bG$ and $\theta x = x$, it is easy to see from the definition of $\hat\phi$ that $\phi |H_{x,0+}^\theta =1$ implies $\hat\phi |K^\theta_+=1$.  (See \S3.9 \cite{ANewYu} for information on the definition of $\hat\phi$.)
\end{proof}

\section{Adapting Lusztig's finite field theory}\label{sec:Lusztigsec}

The main purpose of this section is to prove Proposition \ref{mainfiniteprop} or, in other words, the finite field case of Theorem \ref{maintheorem}.  We also study Lusztig's character $\varepsilon$ (see \cite[page 60]{MR1106911}) and interpret it in terms of determinants of adjoint actions, and then compute it in
 Proposition \ref{Lusztigisdetad}.

\subsection{$\F_q$-rank and its parity}

Let $q$ be a power of an odd prime $p$.
Fix a connected, reductive group $\mathsf{G}$ defined over $\F_q$, and let ${\rm rank}(\mathsf{G})$ denote the $\F_q$-rank of $\mathsf{G}$ and, following Lusztig \cite[page 60]{MR1106911},  define
$$\sigma (\mathsf{G}) = (-1)^{{\rm rank}(\mathsf{G})}.$$

Fix a maximal torus $\mathsf{T}$ in $\mathsf{G}$ that is defined over $\F_q$.  The product $\sigma (\mathsf{G})\, \sigma(\mathsf{T})$ is ubiquitous in the Deligne-Lusztig theory of virtual characters of $\mathsf{G}$ associated to $\mathsf{T}$.

A useful formula for computing this product is 
$$\sigma (\mathsf{G})\, \sigma(\mathsf{T})  = (-1)^{\dim (\mathsf{U}/(\mathsf{U}\cap F\mathsf{U}))},$$ where  $\mathsf{U}$ is the unipotent radical of a Borel subgroup $\mathsf{B}$ of $\mathsf{G}$ containing $\mathsf{T}$, and $F$ is the Frobenius automorphism.  (See \cite[page 60]{MR1106911}.)

The latter formula is also useful in the  form
$$\sigma (\mathsf{G}) = \sigma(\mathsf{T})  (-1)^{\dim (\mathsf{U}/(\mathsf{U}\cap F\mathsf{U}))}$$
as a tool to compute $\sigma (\mathsf{G})$.

Note that when $\mathsf{B}$ (and hence $\mathsf{U}$) is $F$-stable then $\mathsf{T}$ is maximally split.  In this case, the resulting identity $\sigma (\mathsf{G}) = \sigma(\mathsf{T})$ can also be seen as a consequence of the fact that the $\F_q$-rank of $\mathsf{G}$ is  the $\F_q$-rank of any maximally split torus in $\mathsf{G}$.

\subsection{A formula for $\sigma (\mathsf{G})\, \sigma(\mathsf{T})$}

Let $\Phi = \Phi (\mathsf{G},\mathsf{T})$.  
Assume we have fixed a Borel subgroup $\mathsf{B}$ containing $\mathsf{T}$ and having unipotent radical $\mathsf{U}$.  Let $\Phi^+$ be the associated system of positive roots in $\Phi$.

Let $\Gamma = {\rm Gal}(\overline{\F}_q/\F_q)$ be the absolute Galois group of $\F_q$.  Then $\Gamma$ is generated by the Frobenius automorphism $F(x) = x^q$.  Suppose $\mathscr{O}$ is a $\Gamma$-orbit in $\Phi$.
If $a$ is any root in $\mathscr{O}$ and if $d$ is the order of $\mathscr{O}$ then the elements of $\mathscr{O}$ are precisely the elements $a,Fa,F^2a,\dots, F^{d-1}a$.
In this situation, $\F_{q^d}$ is the (minimal) field of definition of $a$ (and all of the elements of $\mathscr{O}$).

Let $-\mathscr{O} = \{ -a\, : a\in \mathscr{O}\}$.  Then $-\mathscr{O}$ is also a $\Gamma$-orbit in $\Phi$.
If $\mathscr{O} = -\mathscr{O}$, we say $\mathscr{O}$ is a {\bf symmetric orbit}.  
Let $(\Gamma\bs \Phi)_{\rm sym}$ denote the set of symmetric orbits.

\begin{proposition}\label{sigGT}
$$\sigma (\mathsf{G})\, \sigma(\mathsf{T})  = (-1)^{\# (\Gamma\bs \Phi)}= (-1)^{\# (\Gamma\bs \Phi)_{\rm sym}}$$ 
\end{proposition}

\begin{proof}
Given a $\Gamma$-orbit $\mathscr{O}$, suppose we fix $a\in \mathscr{O}$ and arrange the elements $a,Fa, \dots , F^{d-1}a$ as equally spaced points on a circle listed in clockwise order.  Now label each element by $+$ or $-$ according to whether it is a positive or negative root.  Let $s(\mathscr{O})$ be the number of sign changes from $-$ to $+$ as one goes clockwise around the circle.

Then the formula
$$\sigma (\mathsf{G})\, \sigma(\mathsf{T})  = (-1)^{\dim (\mathsf{U}/(\mathsf{U}\cap F\mathsf{U}))}$$ 
implies that
$$\sigma (\mathsf{G})\, \sigma(\mathsf{T})  = \prod_{\mathscr{O}\in \Gamma\bs \Phi} (-1)^{s(\mathscr{O})}.$$

Since $s(-\mathscr{O}) = s(\mathscr{O})$ for all $\mathscr{O}$, we have
$$\sigma (\mathsf{G})\, \sigma(\mathsf{T})  = \prod_{\mathscr{O}\in (\Gamma\bs \Phi)_{\rm sym}} (-1)^{s(\mathscr{O})},$$
where $(\Gamma\bs \Phi)_{\rm sym}$ is the set of symmetric orbits.

Now suppose $\mathscr{O}$ is a symmetric orbit and choose a root $a\in \mathscr{O}$.  Let $c$ be the minimal positive integer such that $F^c a=-a$.  Then the elements $a,Fa, \dots, F^{c-1}a,-a,-Fa, \dots, -F^{c-1}a$ are distinct and this sequence must be identical to the sequence $a,Fa,F^2a,\dots, F^{d-1}a$.  In particular, we note that $d=2c$.

In the simplest case, the roots  $a,Fa, \dots, F^{c-1}a$ are all positive and thus the roots $-a,-Fa, \dots, -F^{c-1}a$ are all negative.  In this case, $s(\mathscr{O})=1$.
If $d=2$ or $4$, given an orbit $\mathscr{O}$, one can always choose $a\in \mathscr{O}$ so that the sign pattern is of the latter type.

Now suppose we have an orbit $\mathscr{O}$ of length $d\ge 6$.  In general, we will have some configuration $s_1,\dots, s_c, -s_1,\dots, -s_c$ of $\pm$ signs.
It is easy to check that if we reverse the sign $s_2$ (and hence $-s_2$) then $(-1)^{s(\mathscr{O})}$ is unchanged.  (One only needs to consider $s_1$, $s_2$, $s_3$ and their negatives and check the eight possible sign configurations, which essentially amounts to checking the case of $d=6$.)
One can also adjust $s_3,\dots , s_{c-1}$ so that one essentially reduces to the case of $d=6$.

We deduce that, in general, $s(\mathscr{O})$ is odd if $\mathscr{O}$ is a symmetric root.  Our assertion now follows.
\end{proof}

\subsection{A formula for $\varepsilon_{{}_{\mathsf{T},\theta}}$}

Let $\theta$ be an $\F_q$-automorphism of order two of $\mathsf{G}$ that preserves $\mathsf{T}$.
Let $\mathsf{T}_+ = (\mathsf{T}^\theta)^\circ$ be the identity component of the fixed points of $\theta$ in $\mathsf{T}$.
According to \cite[Section 9]{MR651417}, $\mathsf{T}_+$ is identical to the set of all $t\theta (t)$ with $t\in \mathsf{T}$.  

So if $a\in \Phi$ then $a|\mathsf{T}_+$ is trivial precisely when $\theta a = -a$.

According to \cite[Corollary 5.4.7]{MR1642713}, we have the following relation between the centralizers of $\mathsf{T}$ in $\mathsf{G}$ and its Lie algebra:
$${\rm Lie}(Z_{\mathsf{G}}(\mathsf{T}_+)) = Z_{{\rm Lie}(\mathsf{G})}(\mathsf{T_+}).$$  In addition, we observe that 
$$Z_{{\rm Lie}(\mathsf{G})} (\mathsf{T}_+) = {\rm Lie}(\mathsf{T}) \oplus \bigg( \bigoplus_{a\in \Phi,\, a|\mathsf{T}_+=1}\bfr{g}_a \biggr).$$

Letting  $\Phi_\theta = \Phi (Z_\mathsf{G} (\mathsf{T}_+),\mathsf{T})$, we now have:

\begin{lemma}\label{PhiTheta} Given $a\in\Phi$ then the following are equivalent:
\begin{enumerate}
\item $a\in \Phi_\theta$,
\item $a|\mathsf{T}_+ =1$,
\item $\theta a = -a$.
\end{enumerate}
\end{lemma}

Lusztig defines $\varepsilon = \varepsilon_{{}_{\mathsf{T}}} = \varepsilon_{{}_{\mathsf{T},\theta}}=  \varepsilon_{{}_{\mathsf{T}(\F_q),\theta}}$ by
$$\varepsilon (t) = \sigma (Z_{\mathsf{G}}(\mathsf{T}_+))\, \sigma (Z_{\mathsf{G}}(\mathsf{T}_+)\cap Z_{\mathsf{G}} ( t)^\circ)$$
for $t\in \mathsf{T}^\theta ({\frak f})$.
 
 \begin{proposition}\label{Lusztigisdetad} For all $t\in \mathsf{T}^\theta ( \F_q)$,
 $$\varepsilon_{{}_{\mathsf{T},\theta}} (t) 
 =\prod_{a\in \Gamma\bs \Phi_\theta} a(t).$$
 \end{proposition}
 
 \begin{proof}
Fix  $t\in \mathsf{T}^\theta({\frak f})$.
Then
Proposition \ref{sigGT} implies that
$$\sigma (Z_{\mathsf{G}}(\mathsf{T}_+)) = (-1)^{\# (\Gamma\bs \Phi_\theta)}$$ and
$$\sigma (Z_{\mathsf{G}}(\mathsf{T}_+)\cap Z_{\mathsf{G}} ( t)^\circ) = (-1)^{\# ( (\Gamma\bs \Phi_\theta)_t)},$$
where $(\Gamma\bs \Phi_\theta)_t$ is the set of $\Gamma$-orbits of 
roots $a\in \Phi_\theta$ such that $a(t)=1$.

If we multiply these factors together, then the only orbits that are not counted twice are the orbits of roots $a\in \Phi_\theta$ with $a (t)\ne 1$.  Appealing to Lemma 4.3.1, we observe that, since $\theta (t) = t$ and $\theta a = -a$, we must have $a(t) =-1$ whenever $a(t)\ne 1$.

Therefore, $$\varepsilon_{{}_{\mathsf{T},\theta}} (t) =(-1)^s,$$
where $s$ is the number of $\Gamma$-orbits of roots $a\in \Phi_\theta$ such that $a(t) =-1$.
Since $$(-1)^s = \prod_{a\in \Gamma\bs \Phi_\theta} a(t),$$
our claim follows.
 \end{proof}
 
\subsection{Involutions that do not fix any roots}

The results in this section are not used elsewhere in this paper.  
Our objective is to show how Proposition  \ref{sigGT} and Lemma \ref{PhiTheta} may be applied to provide an alternative proof of
 Lemma 10.5 \cite{MR1106911}.

Let $\mathscr{I}$ be the set of $\F_q$-automorphisms of $\mathsf{G}$ of order two and let $\mathsf{G}(\F_q)$ act on $\mathscr{I}$ by
$$g\cdot \theta = {\rm Int}(g)\circ \theta\circ {\rm Int}(g)^{-1}.$$

Given $\theta\in \mathscr{I}$ and $g\in \mathsf{G}(\F_q)$, then
$\mathsf{T}$ is $(g\cdot \theta)$-stable precisely when $g^{-1}\mathsf{T}g$ is $\theta$-stable.
Therefore studying the elements of $\mathscr{I}$ that stabilize $\mathsf{T}$ is equivalent to studying the $\mathsf{G}(\F_q)$-conjugates of $\mathsf{T}$ that are $\theta$-stable.

We prefer to fix $\mathsf{T}$ and consider  various orbits of elements of $\mathscr{I}$ that stabilize $\mathsf{T}$ without choosing any specific involution $\theta$ as our base point.  
This differs from \cite{MR1106911} and thus some translating between approaches is required.

Lusztig fixes $\theta\in \mathscr{I}$ and defines an associated set $\mathscr{J}_\theta$.
According to \cite[\S10.1]{MR1106911}, the set $\mathscr{J}_\theta$ is identical to the set of $\theta$-stable maximal tori $\mathsf{T}'$ in $\mathsf{G}$ such that $\theta$ does not fix any roots in $\Phi (\mathsf{G},\mathsf{T}')$.

Let us now fix a maximal $\F_q$-torus $\mathsf{T}$ in $\mathsf{G}$ and let $\Phi = \Phi (\mathsf{G},\mathsf{T})$.

The next result is essentially  Lemma 10.5 \cite{MR1106911}.

\begin{lemma}
If $\theta\in \mathscr{I}$ and  $\mathsf{T}\in \mathscr{J}_\theta$ then $\sigma (Z_{\mathsf{G}}(\mathsf{T}_+)) = \sigma (\mathsf{G})$.
\end{lemma}

\begin{proof}
According to Proposition  \ref{sigGT}, it suffices to show that the number of $\Gamma$-orbits in $\Phi -\Phi_\theta$ is even.

Consider a root $a\in \Phi - \Phi_\theta$.   According to Lemma \ref{PhiTheta} and the hypothesis  that $\mathsf{T}\in \mathscr{J}_\theta$,
 the roots $a,-a,\theta a, -\theta a $ must be distinct.
 
We claim that these four roots cannot all lie in a single $\Gamma$-orbit.
Suppose, to the contrary, that they do lie in the same orbit.  Let $\Gamma_a$ be the stabilizer of $a$ in $\Gamma$.  Let $F_a$ be the fixed field of $\Gamma_a$ in $\overline{\F}_q$.  There must exist $\gamma\in \Gamma/\Gamma_a = {\rm Gal}(F_a/\F_q )$ such that $\gamma a = -a$.  Clearly, $\gamma$ has order two, and hence $\gamma$ is the unique element of order two in the cyclic group ${\rm Gal}(F_a/\F_q )$.  Similarly, one can argue that $\gamma a = \theta a$.  Hence, we deduce $\theta a = -a$, which contradicts Lemma \ref{PhiTheta}.

The latter contradiction implies that  the number of  $\Gamma$-orbits in the $\Gamma$-invariant set generated by $a,-a,\theta a, -\theta a $ is even.  But since $\Phi-\Phi_\theta$ is partitioned into $\Gamma$-invariant sets generated by the sets  $\{ a, -a,  \theta a , -\theta a\}$ associated to $a\in \Phi-\Phi_\theta$, our claim follows.
\end{proof}
 
\subsection{Revising Lusztig's formula}

Our main objective in this section is to establish the finite field case of  Theorem \ref{maintheorem} and show that it is consistent with Lusztig's results in \cite{MR1106911}.
We also show in Theorem \ref{newLusztigthm} how Theorem \ref{maintheorem} simplifies over finite fields.

Lusztig gives a symmetric space generalization of  Deligne-Lus\-ztig virtual characters in \cite{MR1106911},  and, in  Theorem 3.3 \cite{MR1106911}, he provides a formula for these generalized virtual characters.  A much simpler formula, \cite[10.6(a)]{MR1106911}, treats the special case of irreducible cuspidal representations.  It is this simpler formula that we revise.

We remark that Theorem 3.3 \cite{MR1106911} has been reformulated and generalized in \cite{MR2798427}, \cite{MR2925798} and \cite{MR3027804}.

As in the previous sections, we fix a connected, reductive group $\mathsf{G}$ defined over $\F_q$.  We also fix a maximal torus $\mathsf{T}$ in $\mathsf{G}$ that is defined and elliptic over $\F_q$.  Fix a character $\lambda$ of $\mathsf{T}(\F_p)$ that is in general position.
Let $\pi = \pi(\lambda)$ be an irreducible, cuspidal representation in the equivalence class associated by Deligne-Lusztig to $(\mathsf{T}(\F_q),\lambda)$.
Then the character of $\pi$ is
$$\sigma (\mathsf{T})\, \sigma (\mathsf{G})\, R_{\mathsf{T}}^\lambda ,$$ where
 $R^\lambda_{\mathsf{T}}$ is the Deligne-Lusztig virtual character associated to\break $(\mathsf{T}(\F_q),\lambda)$.
 
Lusztig's formula 10.6(a) says
 $$\dim (\pi^{\mathsf{G}^\theta (\F_q)}) = \# \left(\mathsf{T}(\F_q)\bs \Theta_{\mathsf{T},\lambda}(\F_q)/ \mathsf{G}^\theta (\F_q)\right),$$ where   $\theta\in \mathscr{I}$, $\pi^{\mathsf{G}^\theta (\F_q)}$ is the space of $\mathsf{G}^\theta (\F_q)$-fixed points in the space of $\pi$, and 
 $$ \Theta_{\mathsf{T},\lambda}(\F_q)= \left\{ g\in \mathsf{G}(\F_q)\, :\, ( g\cdot\theta)(\mathsf{T}) = \mathsf{T},\ \lambda |\mathsf{T}^{g\cdot \theta}(\F_q) = \varepsilon_{{}_{\mathsf{T},g\cdot\theta}}\right\}.$$

 Given a $\mathsf{G}(\F_q)$-orbit $\Theta$ in $\mathscr{I}$, we define
 $$\langle \Theta ,\lambda \rangle_{\mathsf{G}(\F_q)} = \dim {\rm Hom}_{\mathsf{G}(\F_q)^\theta}(\pi,1),$$ where $\theta$ is an arbitrary element of $\mathscr{O}$.
 This invariant is identical to the invariant on the left hand side of Lusztig's formula 10.6(a) according to the following
 standard lemma:
 
 \begin{lemma}
 Assume $(\rho,\mathscr{V})$ be an irreducible, complex representation of a finite group $\mathscr{G}$.
 If $\mathscr{H}$ is a subgroup of $\mathscr{G}$ and $V^{\mathscr{H}}$ is the space of $\mathscr{H}$-fixed points in $\mathscr{V}$ then
 $$\dim {\rm Hom}_{\mathscr{H}}(\rho, 1) = \dim V^{\mathscr{H}}.$$
 \end{lemma} 
 
 \begin{proof}
 Let $\mathscr{V}^*$ be the space of $\C$-linear forms on $\mathscr{V}$ and let $(\tilde\rho, \mathscr{V}^*)$ be the representation given by $$(\tilde\rho (g)\lambda )(v) = \lambda (\rho(g)^{-1}v).$$
 Then $\tilde\rho$ is contragredient to $\rho$.
 
 The space $\mathscr{V}$ has a canonical decomposition $\mathscr{V}^{\mathscr{H}}\oplus \mathscr{V}_{\mathscr{H}}$, into $\mathscr{H}$-submodules, where $\mathscr{V}_{\mathscr{H}}$ is the kernel of the projection $\mathscr{V}\to \mathscr{V}^{\mathscr{H}}$ given by
 $$v\mapsto \frac{1}{|\mathscr{H}|} \sum_{h\in\mathscr{H}} \rho (h)v.$$
 
The contragredient has a similar decomposition $\mathscr{V}^*= \mathscr{V}^{*\mathscr{H}}\oplus \mathscr{V}^*_{\mathscr{H}}$.
A given linear form $\lambda$ on $\mathscr{V}^{\mathscr{H}}$ extends uniquely to a linear form on $\mathscr{V}$ that annihilates $\mathscr{V}_{\mathscr{H}}$.  This yields an embedding of $(\mathscr{V}^{\mathscr{H}})^*$ in $(\mathscr{V}^*)^{\mathscr{H}}$.  Similarly, $(\mathscr{V}_{\mathscr{H}})^*$ embeds in $(\mathscr{V}^*)_{\mathscr{H}}$.  Counting dimensions, we see that $(\mathscr{V}^{\mathscr{H}})^* =\mathscr{V}^{*\mathscr{H}}$ and $(\mathscr{V}_{\mathscr{H}})^*= (\mathscr{V}^*)_{\mathscr{H}}$.

We now have $$\dim \mathscr{V}^{\mathscr{H}} = \dim (\mathscr{V}^{\mathscr{H}})^* = \dim (\mathscr{V}^*)^{\mathscr{H}} = \dim  {\rm Hom}_{\mathscr{H}}(\rho, 1).$$
 \end{proof}
 
 The next result is the finite field case of Theorem \ref{maintheorem}:
 
 \begin{proposition}\label{mainfiniteprop}
$$\langle \Theta ,\lambda\rangle_{\mathsf{G}(\F_q)} = \sum_{\vartheta\sim \lambda} {\rm m}_{\mathsf{T}(\F_q)}(\vartheta ) \ \langle \vartheta , \lambda\rangle_{\mathsf{T}(\F_q)}.$$
\end{proposition}

 \begin{proof}
 The right hand side of Lusztig's formula 10.6(a) is the cardinality of the double coset space
 $$\mathsf{T}(\F_q)\bs \Theta_{\mathsf{T},\lambda}(\F_q)/ \mathsf{G}^\theta (\F_q),$$ where $\theta$ is any element of $\Theta$.
 Let us now fix such an involution $\theta$.
 
The set $ \Theta_{\mathsf{T},\lambda}(\F_q)$ may be partitioned as follows
$$ \Theta_{\mathsf{T},\lambda}(\F_q) = \bigsqcup_{\vartheta\in \mathsf{T}(\F_q)\bs \Theta} \Theta_{\mathsf{T},\lambda}(\F_q)_\vartheta,$$
where
$$\Theta_{\mathsf{T},\lambda}(\F_q)_\vartheta = \left\{ g\in \Theta_{\mathsf{T},\lambda}(\F_q)\, :\, g\cdot \theta\in \vartheta \right\}. $$

(Note that it is elementary to see that both $ \Theta_{\mathsf{T},\lambda}(\F_q)$ and the sets $ \Theta_{\mathsf{T},\lambda}(\F_q)_\vartheta$ are unions of double cosets in $\mathsf{T}(\F_q)\bs \mathsf{G}(\F_q)/ \mathsf{G}^\theta (\F_q)$.)

Given a $\mathsf{T}(\F_q)$-orbit $\vartheta$ in $\Theta$, recall that we write $\vartheta\sim\lambda$ when $\theta' (\mathsf{T}(\F_q))=\mathsf{T}(\F_q)$ and ${\rm Hom}_{\mathsf{T}(\F_q)^{\theta'} }(\lambda,\varepsilon_{{}_{\mathsf{T},\theta'}})$ is nonzero for $\theta'\in \vartheta$.  Equivalently, $\vartheta\sim\lambda$ when $\theta' (\mathsf{T})=\mathsf{T}$ and $\lambda |\mathsf{T}(\F_q)^{\theta'}  =\varepsilon_{{}_{\mathsf{T},\theta'}}$ for $\theta'\in \vartheta$.
We observe that $\Theta_{\mathsf{T},\lambda}(\F_q)_\vartheta$ is nonempty precisely when $\vartheta\sim\lambda$.  So we have
$$ \Theta_{\mathsf{T},\lambda}(\F_q) = \bigsqcup_{\vartheta\sim\lambda } \Theta_{\mathsf{T},\lambda}(\F_q)_\vartheta .$$
Moreover, we observe that if $\vartheta\sim \lambda$ then 
$$ \Theta_{\mathsf{T},\lambda}(\F_q)_\vartheta =\{ g\in \mathsf{G}(\F_q)\, :\, g\cdot \theta \in \vartheta\}.$$

Next, we recall that if $\vartheta\sim \lambda$ then $\langle \vartheta ,\lambda\rangle_{\mathsf{T}(\F_q)}$ is defined as 
$$\langle \vartheta ,\lambda\rangle_{\mathsf{T}(\F_q)} = \dim {\rm Hom}_{\mathsf{T}(\F_q)^{\theta'}}(\lambda ,  \varepsilon_{{}_{\mathsf{T}, \theta'}}),$$ for $\theta'\in \vartheta$.  But this simply means that $\vartheta\sim \lambda$ then $\langle \vartheta ,\lambda\rangle_{\mathsf{T}(\F_q)}=1$.

It now suffices to show that if $\vartheta\sim\lambda$ then the cardinality of 
 $$\mathsf{T}(\F_q)\bs \Theta_{\mathsf{T},\lambda}(\F_q)_\vartheta/ \mathsf{G}^\theta (\F_q)$$
 is 
$${\rm m}_{\mathsf{T}(\F_q)}(\vartheta) = [\mathsf{G}_{\theta'} (\F_q):\mathsf{G}^{\theta'} (\F_q) (\mathsf{G}_{\theta'}(\F_q)\cap\mathsf{T}(\F_q))],$$
where $\theta'\in \vartheta$ and 
$\mathsf{G}_{\theta'} (\F_q)$ is the stabilizer of $\theta'$ relative to the action of $\mathsf{G} (\F_q)$ on $\mathscr{I}$.

We have a bijection
$$\mathsf{G} (\F_q)/\mathsf{G}_\theta (\F_q)\cong \Theta$$ given by $g\mathsf{G}_\theta(\F_q)\mapsto g\cdot\theta$.
This yields a bijection 
$$\mathsf{T}(\F_q)\bs \mathsf{G}(\F_q)/\mathsf{G}_\theta (\F_q)\cong \mathsf{T}(\F_q)\bs \Theta .$$ This bijection can be pulled back via the natural projection
$$\mathsf{T}(\F_q)\bs \mathsf{G}(\F_q)/\mathsf{G}^\theta(\F_q)\to
\mathsf{T}(\F_q)\bs \mathsf{G}(\F_q)/\mathsf{G}_\theta (\F_q)$$
to yield a surjection
$$\mathsf{T}(\F_q)\bs \mathsf{G}(\F_q)/\mathsf{G}^\theta(\F_q)\to
\mathsf{T}(\F_q)\bs \Theta .$$
The constant ${\rm m}_{\mathsf{T}(\F_q)}(\vartheta)$ is identical to the cardinality of the fiber of $\vartheta$ under the latter surjection.
Since this fiber is precisely  $$\mathsf{T}(\F_q)\bs \Theta_{\mathsf{T},\lambda}(\F_q)_\vartheta/ \mathsf{G}^\theta (\F_q)$$
our assertion follows.
\end{proof}

The purpose of the latter result was to provide a formula that unifies the $p$-adic and finite field theories.  However, if one is specifically interested in the finite field theory then the following result is stronger and follows from the previous proof.

\begin{theorem}\label{newLusztigthm}
$$\langle \Theta ,\lambda\rangle_{\mathsf{G}(\F_q)} = \sum_\vartheta {\rm m}_{\mathsf{T}(\F_q)}(\vartheta ),$$ where the sum is over the set of $\mathsf{T}(\F_q)$-orbits $\vartheta$ in $\Theta$ such that
$\theta (\mathsf{T}) =\mathsf{T}$ and $$\lambda (t) = \varepsilon_{{}_{\mathsf{T},\theta}} (t) 
 =\prod_{a\in \Gamma\bs \Phi_\theta} a(t)  ,$$ for some (hence all) $\theta\in \vartheta$ and for all $t\in \mathsf{T}^\theta (\F_q)$.
\end{theorem}

\bibliographystyle{amsalpha}
\bibliography{JLHrefs}

\end{document}